\documentclass[envcountsame]{llncs}
\usepackage[utf8]{inputenc}
\usepackage{amsmath}
\usepackage{amsfonts}
\usepackage{amssymb}
\usepackage{enumitem}

\usepackage{fancyhdr}
\newcommand{\IA}{{\bf IA}}
\newcommand{\IS}{{\bf ISG}}
\newcommand{\BS}{{\bf BSG}}

\newcommand{\A}{\mathfrak{A}}
\newcommand{\s}{\mathfrak{S}}
\newcommand{\B}{\mathfrak{B}}
\newcommand{\M}{\mathfrak{M}}
\newcommand{\G}{\mathfrak{G}}

\title{Implication Algebras and Implication Semigroups of Binary Relations\thanks{This work was supported by the Engineering and Physical Sciences Research Council EP/S021566/1}}
\author{Andrew Lewis-Smith\inst{1}\orcidID{0000-0001-9020-4055
} and Ja{\v s} {\v S}emrl\inst{2}\orcidID{0000-0001-7440-8867}}
\institute{King's College London \email{andrew.lewis-smith@kcl.ac.uk} \url{http://eecs.qmul.ac.uk/profiles/lewis-smithandrewstephen.html}
\and UCL (University College London) \email{j.semrl@cs.ucl.ac.uk} \url{http://www0.cs.ucl.ac.uk/staff/jsemrl/}}

\begin{document}

\maketitle

\begin{abstract}
    Representable implication algebras are known to be axiomatised by a finite number of equations (making the representation and finite representation problems decidable here). We show that this also holds in the context of unary (and binary) relations and present a Stone-style representation theorem. We then show that the (finite) representation  decision problem is undecidable for implication semigroups, in stark contrast with implication algebras.
    
\end{abstract}

\section{Introduction}
The variety of \emph{implication algebras}, so-named by Abbott \cite{10.2307/43679502} and studied by Rasiowa \cite{Rasiowa1974-RASAAA-2}, Diego \cite{diego1965algebras}, and their students, forms the algebraic semantics of the implicational fragment of classical propositional logic. These are Boolean algebras restricted to one operation ($\to$) and a constant ($\top$ or 1). %form an interesting, historically grounded strand of research at the intersection of logic and algebra. 
The variety of \emph{Relation algebras}, alias residuated Boolean algebras with an additional involution operator ($x^\smile$ or `converse of $x$', with $x$ understood as a relation), forms the algebraic semantics of the calculus of relations. By a classical result of Korselt attributed in \cite{LwenheimberMI}, the variety of relation algebras corresponds to the three-variable fragment of classical first-order logic, permitting a study of mathematical logic, particularly set theory \cite{Tarski1987-TARAFO}, via a quantifier-free equational theory. Proofs in this theory consist of simple manipulations of identities, similar to proofs in abstract algebra. This situation contrasts with proofs of standard mathematical logic (or set theory) which can involve complex alternations of quantifiers. %This simplicity is the story with algebraic logic more generally, and 
Meanwhile, implication algebras (having one operation, classical implication) yield an algebraic analysis of the entailment relation between propositions in classical logic. Although informed by different motivations, a certain elegance recommends the study of relation and implication algebras.
\par The present paper considers a fragment of the signature of relation algebras we call \emph{implication semigroups} based on adjoining a semigroup operation ($;$), i.e. relational composition, to the implication algebras of Abbott. When the carrier set of this structure is a set of binary relations, we obtain the fragment of relation algebras consisting of $(S, \to, ;)$, i.e. relation algebras with signature restricted to implication and composition. There are good reasons to examine this signature. For one, it has not been well-explored: practically speaking, most algebraic structures considered in algebraic logic are residuated lattices, groups, or at least monoids -- this can be noticed already in a standard definition of relation algebras, as residuated lattices \cite{Jnsson1993RelationAA} -- where the implication and monoidal operations interact via residuation. Algebras featuring implication and semigroup operations fall out of the mainstream substructural logic literature as the algebra at hand lacks the interdefinability present even in the case of a residuated monoid.

For algebraic logic (particularly relation algebras) the question of whether an algebra has a finite representation looms large. One typically asks whether a given logical system of interest is not just consistent but has finite models, i.e. models we can inspect within finite time or employing finitely many resources. The present paper demonstrates the (finite) representation problem for implication semigroups is undecidable. Our results are curious for two reasons. First, implication semigroups represent, in a sense, a limiting case of substructural logics of implication for which the question of decidability of finite representations, to our knowledge, has not been raised, and certainly not approached from the angle considered here. This suggests a track of further research in what one might call \emph{substructural relation algebras}, exploring the effects of weakening the Boolean base in relation algebra into other algebras of residuation. This is already a current area of research by Peter Jipsen and Nikolaos Galatos \cite{galatos2020weakening} \cite{galatos2020structure}, and has been broached from another angle in \cite{https://doi.org/10.48550/arxiv.2007.13079}, where the signature considered there bears two residuals and a semigroup operation and is in fact a model of the famed Lambek Calculus (thus connecting that algebra to the base system for infinitely many substructural logics). Second, our results contribute to a research programme seeking a better grasp of the consequences for relation algebras when operating in a restricted signature. We are particularly motivated to understand the effect on representability when moving to subsignatures of the standard presentation of a relation algebra \cite{hirsch2002relation}.

%Again a nice bit of context here. Show some interest from the weird intuitionistic logics I don't understand and all that nonsense.

\section{Preliminaries}
\label{sec:prelim}

In this section we present the definitions of the algebraic structures and operators for binary relations. We begin by defining Abbott's implication algebras.

\begin{definition}
An \emph{implication algebra}\footnote{Also known as Tarski algebras.} $\A$ is a pair $(A, \to)$, with $A$ a set and $\to$ a binary operation on $A$ satisfying the following properties:

\begin{enumerate}[label=(\roman*)]
        \item $(a \to b) \to a = a$ (Contraction)
        \item $(a \to b) \to b = (b \to a) \to a$ (Quasi-commutativity)
        \item $a \to (b \to c) = b \to (a \to c)$ (Exchange)
    \end{enumerate}
\end{definition}

Trivially, because the class of implication algebras is equationally definable, it forms a \emph{variety}. We shall refer to this class as $\IA$. Abbott shows a neat property about these in \cite{10.2307/43679502}.

\begin{proposition}[Abbott]\label{prop:abbottid}
Let $\A = (A,\to)$ be an implication algebra. We can implicitly define a constant $1$ as $a \to a$ such that $b \to 1 = 1$ and $1 \to b = b$, for all $b \in A$.
\end{proposition}

This also gives us

\begin{proposition}\label{prop:order}
For an implication algebra $(A, \to)$, we can define a partial order as
$$a \leq b \Leftrightarrow (a \to b) = 1$$
\end{proposition}

\begin{proof}
Let $a+b = (a \to b) \to b$. It is commutative by quasi-commutativity, idempotent by contraction, has $1$ as the top by Proposition~\ref{prop:abbottid}, and can be shown associative (see \cite{10.2307/43679502}[Theorem 12]). So $a+b=b$ forms a partial order. Observe that if $a+b=b$ then $a \to b = a \to ((a \to b) \to b) = (a \to b) \to (a \to b) = 1$. If $a \to b = 1$ then $(a \to b) \to b = 1 \to b = b$.
\qed\end{proof}

\begin{definition}
Let $\top \subseteq X \times X$ be a binary relation. Define $\A(\top) = (\wp(\top), \to)$ where $\to$ is interpreted as proper Boolean implication defined below
$$a \to b = (\top \setminus a) \cup b$$
\end{definition}

One can check that $\A(\top) \in \IA $. Although $\top$ is conventionally an arbitrary maximal relation, this is not the only possible interpretation of the $\to$ operation for binary relations. We say that the implication operator is \emph{absolute} if we require $\top = X \times X$, else we say that it is \emph{relative}.

We say that $\A \in \IA$ is \emph{representable} if and only if it embeds into $\A(\top)$ for some $\top \subseteq X \times X$. The embedding (usually denoted $h$) is called a \emph{representation}. If $\A$ embeds into $\A(\top)$ and $\top$ is over a finite base $X$, then we say $\A$ is also \emph{finitely representable}. 

Another standard presentation of implication algebras is $\A = (A,\to,1)$. However, the constant $1$ can be defined as $a \to a$, for any $a$. Furthermore, the quasi-commutativity axiom is a consequence of the fact that $(a \to b) \to b$ is equivalent to the Boolean join of $a+b$.

\begin{proposition}
\label{prop:jointop}
Let $\A = (A,\to) \in \IA$ be representable via $h$. Then $h((a \to b) \to b) = h(a) \cup h(b)$ and $h(1) = h(a \to a) = \top$, for any $a,b \in A$.
\end{proposition}

\begin{proof}
Since $1 = a \to a$ and $h$ is a representation we get $h(1) = h(a \to a) = h(a) \to h(a) = (\top \setminus h(a)) \cup h(a) = \top$.

By $h$ being a representation, DeMorgan's law, and $a \cap \top = a$ we also have $h((a \to b) \to b) = (h(a) \to h(b)) \to h(b) = \top \setminus((\top \setminus h(a)) \cup b) \cup h(b) = (h(a) \cap (\top\setminus h(b))) \cup h(b) = (h(a) \cup h(b)) \cap ((\top\setminus h(b)) \cup h(b)) = (h(a) \cup h(b)) \cap \top = h(a) \cup h(b)$.
\qed\end{proof}

We now direct our attention to what happens when we add a semigroup operation ($;$) to the signature.

\begin{definition}
An \emph{implication semigroup} $\s$ is a tuple $(S, \to, ;)$, with a carrier set $S$ and $\to, ;$ binary operations on $S$ where

\begin{enumerate}[label=(\roman*)]
        \item $(S, \to)$ is an implication algebra
        \item $(S, ;)$ is a semigroup
        \item $((a \to b) \to b);c = (a;c \to b;c) \to b;c$ (Left quasi-additivity)
        \item $c;((a \to b) \to b) = (c;a \to c;b) \to c;b$ (Right quasi-additivity)
    \end{enumerate}
\end{definition}

The class of implication semigroups will be called $\IS$. Similarly to $\IA$ we also examine structures where the carrier set is a set of binary relations.

\begin{definition}
Let $\top \subseteq X \times X$ be a transitive binary relation. Define $\s(\top) = (\wp(\top), \to, ;)$ where $\to$ is interpreted as proper Boolean implication and $;$ as proper relational composition defined as
$$a ; b = \{(x,z) \mid \exists y \in X: (x,y)\in a, (y,z) \in b\} $$
\end{definition}

Again checking $\s(\top) \in \IS$ is relatively straightforward, note that they are closed under composition due to the transitivity of $\top$. Similarly to $\IA$, $\s \in \IS$ is (finitely) representable if it embeds into $\s(\top)$ for some transitive $\top$ (over a finite base).
 
\section{Basic Theory, Stone Representation, and Decidability for Implication Algebras}
\label{sec:impAl}
We now present the basic theory of implication algebras, the implicational fragment of the implication semigroups discussed in the previous section. We first consider the more general \emph{positive} implication algebras, subsuming the implication algebras. This culminates in a representation theorem for implication algebras, informing our construction in Section~\ref{sec:undec}. \footnote{The representation result for implication algebras appears to have been known to Diego \cite{diego1965algebras}, perhaps Abbott \cite{10.2307/43679502}, but the proof is given in full by Rasiowa in \cite{Rasiowa1974-RASAAA-2}. It was probably known to several others throughout different traditions of algebraic logic.} %We give a sufficiently detailed, streamlined development of the algebra and a modern proof that the implication algebras are isomorphic to an implication algebra of sets. This will facilitate comparison with our later constructions.
%There are two natural examples of implication algebras at hand. The two-element implication algebra, $\mathcal{A}  = (\{1, 0\}, 1, \to)$, with $\to$ operation is defined by equations: $0 \to 0 = 0 \to 1 = 1 \to 1 = 1$ and $1 \to 0 = 0$. The other example is the implication algebra of sets, that is, let $X$ be a set, and $P(X)$ the powerset of $X$. Then $(P(X), X, \to)$ with $Y \to Z = (X \setminus Y) \cup Z$ for all $Y, Z \in P(X)$ is an implication algebra, with the ordering relation given by subset inclusion. In some sense this last example characterises the implication algebras as the Stone-style representation shows).

The axiomatics here are largely in \cite{10.2307/43679502} and \cite{Rasiowa1974-RASAAA-2} with some corrections and modifications. Their presentations of the implication algebras are quite different, Abbott preferring an equational presentation where Rasiowa utilises a quasiequational definition.
%. For instance, we show that including $1$ in the signature of $\IA$ -- standard for presentations of this algebra -- is completely unnecessary, as the constant $1$ is derivable from the axioms. We can than define an ordering on the algebra, and prove basic properties of the algebra before we discuss the theory of implicative filters and finish with a Stone-style representation theorem. 

%Required lemmata for equivalence theorem between prime and irreducible:

\begin{definition}[Rasiowa 2]\label{posimpalg}
A \emph{positive} implication algebra\footnote{Also known as a Hilbert algebra.} (Postive $\IA$) is a pair $(A, \to, 1)$\footnote{With this axiomatisation we cannot omit $1$ from the signature. Alternatively, $1$ could be replaced with $a \to a$ and an extra axiom added as $a \to a = b \to b$.}, a set $A$ and $\to$ satisfying:
\begin{enumerate}[label={(P\arabic*)}]
    \item \label{it:p1} $a \to (b \to a) = 1$
    \item \label{it:p2} $(a \to (b \to c)) \to ((a \to b) \to (a \to c)) = 1$
    \item \label{it:p3} if $a \to b = 1$ and $b \to a = 1$ then $a = b$
    \item \label{it:p4} $a \to 1 = 1$
\end{enumerate}
\end{definition}

Without proof, we state the following lemmas. For proofs, refer to \cite{Rasiowa1974-RASAAA-2}.

\begin{proposition}[Rasiowa 2(1)]\label{modusponens}
In any positive implication algebra, the following condition is fulfilled:
if $a \to b = 1$ and $a = 1$, then $b = 1$. Also, if $a = 1$, then $b \to a = 1$ for any $b \in A$.
\end{proposition}
% \begin{proof}
% Suppose $a \to b = 1$ and $a = 1$. Then $1 \to b = 1$, By \ref{it:p4}, $b \to 1 = 1$
% and using \ref{it:p3}. Similarly, if $a = 1$ then $b \to a = 1$ by \ref{it:p4}.
% \qed\end{proof}

\begin{proposition}[Rasiowa 2.2]\label{posimpinequationtoimp}
For any positive $\IA$ $\mathfrak{A}$, for all $a, b \in A$, we can define a partial order $\leq$ on $A$ as
$$a \leq b \Longleftrightarrow a \to b = 1$$
and $1 = c \to c$ for all maximal $c$ in the poset $(A, \leq)$.
\end{proposition}
% \begin{proof}
% There are four conditions to verify: reflexivity, transitivity, and antisymmetry, and $1$ being the greatest element of the structure. 
% Reflexivity $a \leq a: (a \to (a\to a) \to a)) \to ((a \to (a \to a)) \to (a \to a)) = 1$ by \ref{it:p2}.
% Applying twice \ref{it:p1} and \ref{modusponens} we get
%  $a \to a = 1$, for all $a \in A$. 
% Anti-symmetry, or $a \leq b$ and $b\leq a$ implies $a = b$, is given as follows. So suppose $a \to b = 1$ and $b \to a = 1$. Then by \ref{posimpalg} $a = b = 1$.
% Transitivity: $a \leq b$ and $b \leq c$ implies $a \leq c$. Suppose $a \to b = 1$ and $b \to c = 1$ to show $a \to c = 1$. Applying \ref{modusponens} we obtain $a \to (b \to c) = 1$ and
% $a \to b = 1$. Then by \ref{it:p2} and \ref{modusponens}, we get $a \to c = 1$.
% Last, to show for any $a \in A$, $a \leq 1$: Well, this means we must show $a \to 1 = 1$, but this is true by \ref{it:p4}.
% \qed\end{proof}

\begin{proposition}[Rasiowa 2.3]\label{basiclemma}
 The following hold in any positive implication algebra:
\begin{enumerate}[label={(\arabic*)}]
    \item \label{it:bl1} If $a \leq b \to c$ then $b \leq a \to c$
    \item \label{it:bl2} $a \leq (a \to b) \to b$
    \item \label{it:bl3} $1 \to a = a$
    \item \label{it:bl4} If $b \leq c$, then $a \to b \leq a \to c$
    \item \label{it:bl5} If $a \leq b$  then $b \to c \leq a \to c$
    \item \label{it:bl6} $a \to (b \to c) = b \to (a \to c)$
\end{enumerate}
\end{proposition}

\begin{proposition}[Distributivity]\label{distr} In any (positive) implication algebra $\mathfrak{A} = (A, \to, 1)$, we have
$a \to (b \to c) = (a \to b) \to (a \to c)$
\end{proposition}
\begin{proof}
%$[(a \to b) \to (a \to c)] \to [a \to (b \to c)] = [(a \to b) \to (a \to c)] \to [b \to (a \to c)] = [(a \to c) \to (a \to b)) \to [b \to (a \to b)] = [(a \to c) \to (a \to b)] \to 1 = 1$.
We have $b \leq a \to b$ by \ref{it:p1} and Proposition~$\ref{posimpinequationtoimp}$. Applying Proposition~\ref{basiclemma}\ref{it:bl5}\ref{it:bl6}, we get
$(a \to b) \to (a \to c) \leq b \to (a \to c) = a \to (b \to c)$. So, $a \to (b \to c) = (a \to b) \to (a \to c)$ follows from \ref{it:p2} and Proposition~\ref{posimpinequationtoimp}.

The proof that distributivity holds in implication algebras is found in \cite[Theorem 5]{10.2307/43679502}.
\qed\end{proof}

We now show that the class of implication algebras lies below the class of positive implication algebras. Although the following proposition is not in Abbott or Rasiowa, it is latent in the published results concerning implicative, positive implication, and implication algebras.

\begin{proposition}\label{2.1f}
 Any implication algebra $(A, \to)$ is a positive implication algebra. 
\end{proposition}
\begin{proof}
% First, to show implication algebras are positive implication algebras.
\ref{it:p1} follows from the exchange axiom and Proposition~\ref{prop:abbottid}, more specifically $a \to (b \to a) = b \to (a \to a) = b \to 1 = 1$. For \ref{it:p2} follows from Proposition~\ref{distr} and Proposition~\ref{prop:abbottid}. For \ref{it:p3} see that by Proposition~\ref{prop:order} we have the anti-symmetry for the partial order in implication algebras. Finally \ref{it:p4} follows directly from Proposition~\ref{prop:abbottid}.
%Thus $a \to a = \top$ holds. To prove $$((a \to b) = \top), ((b \to c) = \top) \Rightarrow \space ((a \to c) = \top)$$ we simply apply \ref{inequationtoimp}, specifically using transitivity of the inequations.  Conditions (i3 ),
%(i4) are the same as \ref{it:p3} and \ref{it:p4}, respectively. The second statement is obvious.
\qed\end{proof}

\begin{proposition}
    Any positive implication algebra $(A, \to)$ satisfying
$$(a\to b) \to a = a$$
    for all $a,b \in A$ is an implication algebra.\footnote{Note that the contraction identity is not provable from the axioms \ref{it:p1}--\ref{it:p4}, a counterexample can be found using Mace4.}
\end{proposition}

\begin{proof}
% The first two axioms of positive implication algebras, \ref{it:p1} and \ref{it:p2}, we demonstrate. . . Alternatively, we can get it from exchange, the sixth item of \ref{basiclemma}: $a \to (b \to a) = b \to (a \to a) = b \to 1 = 1$. Next: if $(a \to b) = 1$ and $(b \to a) = 1$, then $a = b$, follows from the \ref{posimpinequationtoimp}, specifically the antisymmetry of the partial order relation. The last axiom, $a \to 1 = 1$, follows again follows from the \ref{posimpinequationtoimp}, that $1$ is the top element of the ordering.
To show the other direction, let $(A, \to,1)$ be a positive implication algebra satisfying $(a \to b) \to a = a$. The first axiom of implication algebras $(a \to b) \to a = a$ we have already assumed adjoined to the algebra, and the third axiom, $a \to (b \to c) = b \to (a \to c)$, is found in Proposition~\ref{basiclemma}\ref{it:bl6}. To show the second axiom: $(a \to b) \to b = (b \to a) \to a$, we note $a \to b \leq 1 = (b \to b) = (b \to a) \to (b \to b) = b \to ((b \to a) \to  b)$ by Proposition~\ref{posimpinequationtoimp} and Proposition~\ref{basiclemma}\ref{it:bl6}. By Proposition~\ref{basiclemma}\ref{it:bl1} we have $b \leq (a \to b) \to ((b \to a) \to  b)$ and thus by Proposition~\ref{basiclemma}\ref{it:bl1} and \ref{it:bl3} we get $(a \to b) \to b \leq ((a \to b) \to (a \to b)) \to ((b \to a) \to  b) = 1 \to ((b \to a) \to  b) = ((b \to a) \to  b)$. By a completely analogous argument, $(a \to b) \to b \leq (b \to a) \to a$. Hence $(a \to b) \to b = (b \to a) \to a$ as desired.
% by Proposition~\ref{basiclemma} we have $a \to (a \to b) \leq b$.
% Using Proposition~\ref{basiclemma} again, if $b \leq c$, then $a \to b \leq a \to c$, and since by Proposition \ref{basiclemma} we have $a \to (b \to c) = b \to (a \to c)$ and $(a \to b) \to a = a$, we get $(b\to a) \to a \leq (b\to a) \to (a \to b) \to b) = (a \to b) \to  (b \to a) \to b) = (a \to b) \to b$.
% Thus $(b \to a) \to a \leq (a \to b) \to b$. By a completely analogous argument, $(a \to b) \to b \leq (b \to a) \to a$. Hence $(a \to b) \to b = (b \to a) \to a$ as desired.
\qed\end{proof}

In anticipation of the Stone-style representation theorem, we define some required notions like that of an implicative filter.

\begin{definition}[Abbott]\label{defoffilter}
An implicative filter of a (positive) implication algebra $\mathfrak{A} = (A,\to)$ is a subset $F \subseteq A$ such that: 

\begin{enumerate}[label=(\roman*)]
    \item $1 \in F$
    \item if $a \in F$ and $a \to b \in F$ then $b \in F$ 
\end{enumerate}
\end{definition}

% \begin{definition}[Abbott]In any implication algebra
% $\mathfrak{A} = (A, \to)$, we can define a join operation, for all $a, b$:
% $$a + b = (a \to b) \to b$$
% \end{definition}
% Hence the implicative algebras can be understood as semilattices.%\footnote{May derive the rest of the equations.}
\begin{definition}\label{prime}
We say that an implicative filter $F$ is \emph{proper} if $F \neq A$. We say that a proper implicative filter is \emph{irreducible} if it is not the intersection of two proper implicative filters distinct from it, or formally: $F$ is irreducible if for any two proper implicative filters $F_{1}$, $F_{2}$ such that $F = F_{1} \cap F_{2}$, either $F = F_{1}$ or $F = F_{2}$. Finally, a proper implicative filter $F$ is said to be \emph{prime} if $a + b  \in F$ (or equivalently  $(a \to b) \to b \in F$) implies that either $a \in F$ or $b \in F$, for all $a,b \in A$.
\end{definition}

%We record a basic fact about implicative filters, which has as a consequence that if $G$ is a class of implicative filters, and $C$ a chain in $G$, that $C$ has an upper bound in $G$.

% \begin{lemma}\label{closureoffilters}
% Implicative filters are closed under unions of chains, and non-empty classes of implicative filters are closed under intersections. \marginpar{what is a class of implicative filters? Surely the empty ones are also vacuously closed under intersection? Also do we need this Lemma?}
% \end{lemma}
% \begin{proof}
% Let $G = \{F_{0} \subset F_{1} \subset F_{2} \subset \ldots \}$ be a strictly increasing chain of implicative filters on an implicative algebra $\mathfrak{A}$. Let $\bigcup G$ be the union of this chain of implicative filters. $1 \in \bigcup G$ for otherwise $1 \notin F_{i}$ for some $i$ in the sequence, which is impossible by construction. Similarly, if $a$ and $a \to b \in \bigcup G$ then $b \in \bigcup G$ for otherwise, for some $i$ in the sequence $a$ and $a \to b \in F_{i}$ but $b \notin F_{i}$. Similarly, this is impossible, as every $F$ in the sequence is closed under implication by definition.
% \qed\end{proof}

The proof of the Stone-like Representation theorem follows the following steps. For proofs, refer to \cite{Rasiowa1974-RASAAA-2}.

%\begin{definition}\label{princimpfilterdef}
%Given an implication algebra  $\mathfrak{A}$, we call $F(a) = \{x \in A: a \leq x\}$, with $F(a) \subseteq A$, a \emph{principal} implicative filter. $F(a)$ is the least or smallest implicative filter containing $a$.
%\end{definition}
%\begin{note}
%$F(a)$ is indeed an implicative filter by Lemma %\ref{removeoutera} and Proposition \ref{partial-order}.
%\end{note}

\begin{proposition}[Rasiowa 1.8]\footnote{Rasiowa states this result for implicative algebras, the weakest algebra she considers in her text. Since all implication algebras are positive implication algebras, and all positive implication algebras are implicative algebras, we can specialise her result for the present case.} 
\label{1.8} If in any (positive) implication algebra $\mathfrak{A} = (A, \to)$ one of the following
conditions is satisfied for all $a, b, c \in A$:
\begin{enumerate}[label=(F\arabic*)]
    \item $(a \to (b \to c)) \to a \to b) \to (a \to c)) = 1$
    \item $(a \to b) \to (a \to (b \to c)) \to (a \to c)) = 1$
\end{enumerate}

then for every implicative filter $F$ in $\mathfrak{A}$ and for every $a \in A$, the set $F_a*
= \{x \in A: a \to x \in F\}$ is an implicative filter. If, moreover, for all $a,b \in A: a \to (b \to a) = 1$, then $F_a*$ is the \emph{least} implicative filter containing $F$ and $a$.
\end{proposition}
% \begin{proof}
% We prove the claim for positive implication algebras, and since implication algebras are subsumed under positive implication algebras by Proposition \ref{2.1f}, we will have the result for implication algebras as well.
% So, $1 \in F*$ as we have $(a \to 1) = 1$ by Definition \ref{posimpalg} and $1 \in F*$ by Definition \ref{defoffilter}. 
% If $b , b \to c \in F*$, then if $a \to b$, we also have $a \to (b \to c) \in F$. By (F2) or (F3) and Definition \ref{defoffilter}, $a \to c \in F$, i.e. $c \in F*$. Thus $F*$ is an implicative filter. By the fact that $a \to a = 1$ (see Proposition \ref{posimpinequationtoimp}) and Definition \ref{defoffilter}, $a \in F*$. If $b \in F$, then using $b \to (a \to b) = 1 = b \to (a \to b)\in F$, so $a \to b \in F$, thus $b \in F*$. Thus $F \subset F*$. Now let $F_0$ be an implicative filter such that $F \subset F_0$
% and $a \in F_0$. If $b \in F*$, then $a \to b \in F \subset F_0$. Since $a \in F_0$, by Definition \ref{defoffilter} $b \in F_0$. And $F* \subset F_0$.
% \qed\end{proof}

    \begin{proposition}[Rasiowa 3.4]\label{lemmaimpalgebraposimp}
If $(A, \to)$ is a (positive) implication algebra, then for every implicative filter F and for every element $a \in A$ the set $F_a* = \{x \in A : a \to x \in F \}$ is the least implicative filter containing $F$ and $a$.
\end{proposition}

\begin{proposition}[Rasiowa 6.1]\label{equivbetweenprimeandirr}
An implicative filter in an implication algebra is prime if and only if it is irreducible.
\end{proposition}

\begin{lemma}[Rasiowa 1.4]\label{irrimplicativefilter}
If $F_{0}$ is an implicative filter in an implicative algebra $\mathfrak{A}$ such that $a_{0} \notin F_{0}$ for some $a_{0} \in \mathfrak{A}$ then there exists an irreducible implicative filter $G$ such that $F_{0} \subset G$ and  $a_{0} \notin G$. 
\end{lemma}
% \begin{proof}
%     Let $G$ be the set of all implicative filters $F \in \mathfrak{A}$ containing $F_{0}$, ordered by inclusion, such that $a_{0} \notin F$. By Lemma \ref{closureoffilters}, every chain in $G$ has an upper bound, and by Zorn's lemma we have a maximal element $F* \in G$. Obviously, $F_{0} \subset F*$ and $a_{0} \notin F*$. Now suppose $F* = F_{1} \cap F_{2}$ for some proper implicative filters $F_{1}, F_{2}$, and thus $F_{0} \subset F_{1}$ and $F_{0} \subset F_{2}$, and either $a_{0} \notin F_{1}$ or  $a_{0} \notin F_{2}$. Hence either $F_{1} \in G$ or  $F_{2} \in G$. Since $F*$ maximal in $G$ and $F* \subset F_{1}$ or  $F* \subset F_{2}$, either $F* = F_{1}$ or $F* = F_{2}$.
% \qed\end{proof}

Immediately, by Lemma~\ref{irrimplicativefilter} and Proposition~\ref{equivbetweenprimeandirr} we have:

\begin{corollary}\label{primeimplicativefilter}
If $F$ is an implicative filter in an implicative algebra $\mathfrak{A}$ such that $a \notin F$ for some $a \in \mathfrak{A}$ then there exists a prime implicative filter $G$ such that $F \subset G$ and  $a \notin G$.
\end{corollary}

This next corollary we prove, as it is not found in any of the literature cited above and is required for the representation theorem.

\begin{corollary}\label{primeextensionfailingb}
Let $F$ be an implicative filter of an implication algebra $\mathfrak{A} = (A, \to)$ such that $a \to b \notin F$ for some $a, b \in A$. Then there exists a prime implicative filter $G : F \subseteq G$ such that $a \in G$ and $b \notin G$.\end{corollary}
\begin{proof}
%Let $\mathfrak{A} = (A, \to)$ be an implication algebra, and we prove the contrapositive: suppose there is no irreducible implicative filter $F \in \mathfrak{A}$ such that $a \in F$ and $b \notin F$. Then \emph{every} irreducible implicative filter $F$ of $\mathfrak{A}$ such that $a \in F$ implies $b \in F$; hence take two proper implicative filters, $G$, $H$ such that $F = G \cap H$, supposing that $a \in F$.  Then $b \in F$, and either $F = G$ or $F = H$. We show: $a \leq b$ Note $b \nleq a$, as .
%\begin{itemize}
%    \item Case 1. Suppose $F = G$. Then $F = G = G \cap H$ thus $G \subset H$, and so $F \subset H$; by the above definition, $a \in F$ gives us $F(a)$ (principal filter generated by $a$). Since $F \subset H$, and $b \in F$, we have $b \in H$, hence $H(b)$ by the definition of principle filter again. %Now $b \in F(a)$, 
 %   Thus $H(b) \subset F(a)$, and this implies $a \leq b$. 
  %  \item Case 2. Suppose $F = H$. $F \subset G$; We show $F(a) \subset G(b)$. Argument analogous.
%\end{itemize}
%Hence we have $a \leq b$, as desired.
%PROOF TO BE CHECKED BY SOMEONE SANE.

Let $F_a*$ be the implicative filter generated by the filter $F$ and $a$. Suppose that $a \to b \notin F$. If $b \in F_a*$, then we have $a \to b \in F$ by the definition of $F_a*$. This contradicts our assumption that $a \to b \notin F$; hence $b \notin F_a*$, and applying Corollary~\ref{primeimplicativefilter} for $F_a*$ and $b$ we have a prime filter $G$ such that  $F_a* \subseteq G$ and $b \notin G$. Clearly, $a \in G$ and $F \subseteq G$.
\qed\end{proof}

We have then, as an immediate corollary from Corollary~\ref{primeextensionfailingb} and Proposition~\ref{equivbetweenprimeandirr}, the following:

\begin{corollary}
Let $F$ be an implicative filter of an implication algebra $\mathfrak{A} = (A, \to)$ such that $a \to b \notin F$ for some $a, b \in A$. Then there exists an irreducible implicative filter $G : F \subseteq G$ such that $a \in G$ and $b \notin G$.
\end{corollary}

Finally, the culminating representation theorem. Rasiowa presents this for irreducible implicative filters \cite{Rasiowa1974-RASAAA-2}, which given her equivalence result, one can also state using prime implicative filters, or maximal implicative filters.

\begin{theorem}[Rasiowa 7.1]\label{Stone}
For any implication algebra $\mathfrak{A} = (A, \to)$, there is a monomorphism $h$ from $\mathfrak{A}$ to $(\wp(X), \to)$ of an arbitrary space $X$ with $|X| \geq A$.
\end{theorem}

From this it follows that every implication algebra is isomorphic to an implication algebra of sets. Since the focus of the present paper is on representations, we note a corollary from this last result \cite{10.2307/43679502,diego1965algebras,Rasiowa1974-RASAAA-2}:

\begin{corollary}
For any implication algebra $\mathfrak{A}$, if $A$ is finite, then $\mathfrak{A}$ has a finite representation.
\end{corollary}

\begin{proof} Let $\mathfrak{A}$ be a finite implication algebra. Then by Theorem \ref{Stone} $\mathfrak{A}$ is monomorphic to the algebra $\mathfrak{A}\prime$ under $h$, where $\mathfrak{A}\prime = (\wp(X), \to)$ , an implication algebra of sets. Now if $|X| = A$ then $|\wp(X)| = 2^{|A|}$, and thus finite. That means %$\mathfrak{A}\prime$ may not be finite, but as $h$ is a monomorphism, $|Domain(h)| \leq |Range(h)|$\marginpar{The powerset should be $2^n$ not n if size of $X$ is $n$?} %for all monomorphisms between structures (and injections between sets)
 $h(\mathfrak{A})$, the subalgebra of $\mathfrak{A}\prime$ under $h$, is finite. So we have $h(\mathfrak{A})$ is a finite implication algebra (induced by $h$ and $\mathfrak{A}$), and hence $h$ is a finite representation of $\mathfrak{A}$.
\qed\end{proof}

Now, the focus of the rest of the paper revolves around the (finite) representation decision problem for implication semigroups. In the case of $\IA$, this is defined as follows:

\begin{definition}
\emph{The (finite) representation decision problem for implication algebras} is a decision problem that takes an implication algebra with a (finite) carrier set as input. The algebra is a yes instance if and only if it is (finitely) representable.
\end{definition}

Closing this section, we note:
\begin{corollary}
$\IA$ is finitely axiomatisable.
\end{corollary}

% and hence as a consequence:

\begin{corollary}
The (finite) representation problem for $\IA$ is decidable.
\end{corollary}

%Next we prove something which, to our knowledge, has not previously been shown for theimplication algebras.

%\begin{theorem}
%Everyimplication algebra can be embedded into to an %atomic Implication algebra. 
%\end{theorem}

%\begin{proof}

%\qed\end{proof}

\section{Undecidability Results for Implication Semigroups}
\label{sec:undec}

In this section we build on results from \cite{hirsch2012undecidability,neuzerling2016undecidability,hirsch2021undecidability} to show undecidability of some decision problems for . We begin by defining the representation and the finite representation decision problems.

\begin{definition}
The \emph{(finite) representation decision problem for implication semigroups} is a decision problem that takes an implication semigroup with a finite carrier set as input. The semigroup is a yes instance if and only if it is (finitely) representable.
\end{definition}

As we mention in Section~\ref{sec:prelim}, whether a structure is representable, also depends on our interpretation of the constant $1$. Here we show that the (finite) decision problem is undecidable in both cases.

% Furthermore, for the conventional case where $1$ is an arbitrary transitive relation, we also show that the equational theory decision problem for [finitely] representable implication semigroups, defined below, is undecidable.

% \begin{definition}
% The \emph{equational theory decision problem} for [finitely] representable implication semigroups is a decision problem that takes a pair of $\langle \to , ; \rangle$-terms $(t_1,t_2)$ as input. $(t_1,t_2)$ is a yes instance if and only if $t_1 = t_2$ is true for all [finitely] representable implication semigroups.
% \end{definition}

\subsection{Representation Problem with Absolute Implication}

We begin by examining the case with absolute implication, i.e. we require $\top = X \times X$ for some (finite) base $X$.

\begin{definition}
An \emph{implication monoid} $\M = (M, 1', \to, ;)$ is an algebra where $(M,\to,;)$ is an implication semigroup and $1'$ is the monoidal identity for $;$. For some transitive and reflexive $\top \subseteq X \times X$, we define $\M(\top) = (\wp(\top), 1', \to, ;)$ where $\to,;$ are proper relational implication and composition respectively and $1'$ is the proper relational identity for $X$ defined as $1' = \{(x,x) \mid x \in X\}$.
\end{definition}

%%%%%%%%%%%%%%%%%%%%%%%%%
% Note about the above definition. Since we are talking about the absolute complementation, maybe \top = X \times X? IDK lmk what you think, not a big deal.
%%%%%%%%%%%%%%%%%%%%%%%%%

In \cite[Section 4]{hirsch2012undecidability} a construction of a Boolean monoid from a square cancellative partial group $\G$ is given. Its implication monoid reduct is denoted $\M(\G) = (M,1',\to,;)$. By \cite[Proposition 5.1, Example 6.2]{hirsch2012undecidability} $\M(\G)$ is representable (over a finite base) if and only if $\G$ embeds into a (finite) group.

From the fact that both the group and the finite group embedding problems are undecidable \cite{evans1953embeddability} for finite structures it follows that the (finite) representation decision problem is undecidable. Thus if we prove that the $\IS$ reduct of $\M(\G)$ is (finitely) representable if and only if $\M(\G)$ is representable, we have shown that the (finite) representability is undecidable. The right to left implication is trivial. But we must examine the case where we relax the requirement where we represent $1'$ as the true relational identity, and show that this is still sufficient for the structure to remain (finitely) representable with $1'$ taken as the true relational identity.

Suppose we have an embedding $h$ from $\M(\G)$ to $\s(\top)$, i.e. an injective mapping that preserves $\to, ;$, but not necessarily $1'$.

\begin{lemma}
\label{lem:idrev}
If $(x,y) \in h(1')$ then $(y,x) \in h(1')$.
\end{lemma}

\begin{proof}
Suppose $(y,x) \not \in h(1')$. That means that $(y,x) \in h(1' \to 0) = h(\overline{1'})$. By composition of $(x,y) \in h(1')$ and $(y,x) \in h(\overline{1'})$ we get that $(x,x) \in h(\overline{1'})$ and by composing that with $(x,y) \in h(1')$ we have that $(x,y) \in h(\overline{1'})$. As $(x,y) \in h(1')$ and $(x,y) \in h(\overline{1'}) = h(1' \to 0)$, we also have $(x,y) \in h(0)$. By a series of compositions we also get that $(y,x) \in h(0)$ and because $0 \leq 1'$ we also get $(y,x) \in h(1')$ and we've reached a contradiction.
\qed\end{proof}

\begin{lemma}
\label{lem:eqRel}
$h(1')$ is an equivalence relation.
\end{lemma}

\begin{proof}
By Lemma~\ref{lem:idrev} we have that $h(1')$ is symmetric. Furthermore, since all $(x,x) \in h(\top)$ there must exist a $z$ witnessing $1';\top = \top$. Thus $(x,z) \in h(1')$ and $(z,x) \in h(1')$ and we compose that to get $(x,x) \in h(1')$, so $h(1')$ is reflexive. Finally, as $1' = 1';1'$ we also have that $h(1')$ is transitive.
\qed\end{proof}

\begin{lemma}
\label{lem:eqClassEqLabel}
For all $x,x',y,y' \in X$ where $(x',x),(y,y') \in h(1')$ we have for all $a \in \mathfrak{M}(\mathfrak{G})$ that $(x,y) \in h(a) \Leftrightarrow (x',y') \in h(a)$.
\end{lemma}

\begin{proof}
If $(x,y) \in h(a)$ we have $(x,y') \in h(a)$ by $(x',x),(y,y') \in h(1')$ and the composition of $1';a;1'=a$. By Lemma~\ref{lem:idrev}, we also have $(x,x'),(y',y) \in h(1')$ so similarly if $(x,y') \in h(a)$ then $(x,y) \in h(a)$.
\qed\end{proof}

\begin{theorem}
The (finite) representation decision problem for $\IS$ is undecidable when $\to$ is interpreted as absolute implication.
\end{theorem}

\begin{proof}
As $h(1')$ is an equivalence relation by Lemma~\ref{lem:eqRel}, so we can define $h': \M(\G) \to X/h(1')$ where
$$h'(a) = \{([x]_{h(1')}, [y]_{h(1')}) \mid (x,y) \in h(a)\}$$
and show that $h'$ is indeed an embedding of $\M(\G)$ into $\M(X \times X)$.

By Lemma~\ref{lem:eqClassEqLabel}, we know that if $(x,y) \in h(a)$ then for any $x' \in [x]_{h(1')}, y' \in [y]_{h(1')}$ we have $(x',y') \in h(a)$. 

Take any $a \leq b$. Then there exists $(x,y) \in h(a)\setminus h(b)$. From this follows $([x]_{h(1')}, [y]_{h(1')}) \in h'(a)$ and if it were the case that $([x]_{h(1')}, [y]_{h(1')}) \in h'(b)$ that would mean that there exist some $(x',y') \in h(b)$ with $(x,x') \in h(1')$ and $(y',y)\in h(1')$ and that would also means that $(x,y) \in h(b)$. Thus $h'$ is injective.

Every composition is witnessed by the equivalence class of the witness for the composition in $h$ and if $(x,y) \in h(a)$ and $(y',z) \in h(b)$ with $y' \in [y]_{h(1')}$ we also have $(y,y') \in h(1')$ and thus we have the composition $(x,z) \in h(a;1';b) = h(a;b)$. Thus $h'$ represents $;$ correctly. Finally $1'$ is represented correctly as a pair of equivalence classes is in $h'(1')$ if and only if they are the same equivalence class.

Thus we have shown that if we have an embedding of $\M(\G)$ into $\s(X \times X)$ then we also have an embedding of $\M(G)$ into $\M(X' \times X')$ where $X' = X/h(1')$. Furthermore if $X$ is finite, so is $X'$. Trivially if $\M(G)$ embeds into $\M(X\times X)$ it also embeds into $\s(X\times X)$ via the same embedding. This, together with the results presented in \cite{hirsch2012undecidability} shows that the (finite) representation decision problem for $\IS$ is undecidable.
\qed\end{proof}

\subsection{Representation Problem with Relative Implication}

Now we show the same result for relative implication.

\begin{definition}
A \emph{Boolean semigroup} is a tuple $\B = (B,0,1,-,+,;)$ is an algebraic structure where $S$ is a carrier set
\begin{enumerate}[label=(\roman*)]
    \item $(B,0,1,-,+)$ is a Boolean algebra
    \item $(B,;)$ is a semigroup
    \item $;$ is additive over $+$
    \item $0;a = a;0 = 0$
\end{enumerate}
\end{definition}

Similarly to , we denote the class of Boolean semigroups $\BS$ and we say that a Boolean semigroup is representable if and only for some transitive $\top \subseteq X \times X$ it embeds into $\B(\top) = (\wp(X\times X), \emptyset, \top, -, +, ;)$ where $-a$ is interpreted as proper Boolean negation $\top \setminus a$, $+$ is interpreted as proper Boolean join $\cup$ and $;$ is interpreted as proper relational composition.
 
The (finite) representation problem for Boolean semigroups is defined analogous to that for implication semigroups. \cite[Theorem 11.2]{hirsch2021undecidability} shows that the representation problem for Boolean semigroups is undecidable and \cite[Theorem 2.5]{neuzerling2016undecidability} shows that the finite representation problem for Boolean semigroups is undecidable. From this we show that the (finite) representation problem for implication semigroups is also undecidable.

Note that the above results require an operation $\cdot$ to be defined in the signature, but much like $\top$ in $\IA$, $\cdot$ is term definable for $\BS$ as $a \cdot b = -(-a + -b)$.

\begin{lemma}
\label{lem:zeroRep}
Let $\s = (S, \to, ;)$ be an implication semigroup that contains some element $0$ such that for all $a \in S$ we have $0 \leq a$ and $0;a=a;0=0$. If $\s$ is representable via some representation $h$ then there exists a representation $h'$ of $\s$ where $h'(0) = \emptyset$. If $h$ is defined over a finite base, so is $h'$.
\end{lemma}

\begin{proof}
Let $h(\s)$ be the proper structure defined by $h$ for some $\top \subseteq X \times X$. As $h$ is a representation, there exists for every pair $a \not \leq b \in S$ a \emph{discriminator pair} $(\iota,o) \in \top$ such that $(\iota,o) \in h(a) \setminus h(b)$.

Define $X^{\iota,o}$ as
$$X^{\iota,o} = \bigg\{x \in X \mid \Big(x = \iota \vee (\iota,x) \in \top\Big) \wedge \Big(y = o \vee (y,o) \in \top\Big) \bigg\}$$
$\top^{\iota,o}$ as $\top \cap (X^{\iota,o} \times X^{\iota,o})$, and a mapping $h^{\iota,o}: \s \to \s(\top^{\iota,o})$ where $h^{\iota,o}(a) = h(a) \cap \top^{\iota,o}$.

First observe that $h^{\iota,o}(0) = \emptyset$. Suppose that there was a pair $(x,y) \in h^{\iota,o}(0)$. If $x=\iota$, we have $(\iota,y) \in h^{\iota,o}(0)$, else $(\iota,x) \in h(1) = \top$ and thus $(\iota, y) \in h(0)$ since $1;0=0$ and $h$ preserves composition. Similarly if $y = o$ we get $(\iota, o) \in h(0)$, else by composing $(\iota,y) \in h(0)$ with $(y,z) \in h(1)$ we get $(\iota,o) \in h(0)$. Since $b \geq 0$ that would mean $(\iota, o) \in h(b)$ that contradicts the fact that $(\iota,o)$ is a discriminator pair for $a \not \leq b$.

Now let us check that $h^{\iota,o}$ preserves composition. Suppose $(x,y) \in h^{\iota,o}(a;b)$. This means that there exists $z \in X$ such that $(x,z) \in h(a)$ and $(y,z) \in h(b)$. If $x=\iota$, we trivially have $(x,\iota) \in h(1) = \top$. Else, by composing $(\iota,x) \in h(1)$ and $(x,y) \in h(a)$ we get $(\iota,y) \in h(1) = \top$ as $1;a \leq 1$. Similarly $(y,o) \in \top$ and thus $y \in X^{\iota,o}$. Thus we have $(x,z) \in h^{\iota,o}(a), (y,z) \in h^{\iota,o}(b)$ and we have shown $h^{\iota,o}(a;b) \subseteq h^{\iota,o}(a);h^{\iota,o}(b)$. The fact that $h^{\iota,o}(a;b) \supseteq h^{\iota,o}(a);h^{\iota,o}(b)$ follows from $(x,y),(y,z) \in \top^{\iota,o}$ then $x,z \in X^{\iota,o}$ and we have $(x,z) \in \top^{\iota,o}$. Thus $h^{\iota,o}$ preserves composition.

We have $h^{\iota,o}(a) \neq h^{\iota,o}(b)$ as $(\iota,o) \in \top^{\iota,o}$. The operation $\to$ is preserved by $h^{\iota,o}$ as for all $(x,y) \in \top^{\iota,o}$ it holds $(x,y) \in h(a) \Longleftrightarrow (x,y) \in h^{\iota,o}(a)$. Finally, $|X^{\iota,o}| \leq |X|$. Thus we conclude that $h^\{\iota,o\}$ is a homomorphism for $\s$ that discriminates the pair $a \not \leq b$.

Now let us pick for every $a \not \leq b$ a $\delta(a,b) = (\iota,o)$ such that $(\iota,o) \in \top$ is a discriminator pair for $a \not \leq b$ and let $\dot\cup$ denote a disjoint union. A mapping
\begin{gather*}
    h': S \to \wp\left(\dot\bigcup_{a \not \leq b \in S}\top^{\delta(a,b)} \right)\\
    h'(c) = \dot\bigcup_{a \not \leq b \in S}h^{\delta(a,b)}(c)
\end{gather*}
still represents $;, \to$ correctly, discriminates all pairs $a \not \leq b$ (i.e. is injective), which makes it a representation. Furthermore, $h'(0) = \emptyset$ and the size of its base is bounded by $|S|^2|X|$.
\qed\end{proof}

\begin{theorem}
The (finite) representation decision problem for implication semigroups is undecidable.
\end{theorem}

\begin{proof}
We show this by proving that $\B \in \BS$ is representable if and only if its $\langle \to, ; \rangle$-reduct $\s$ is representable. The left to right implication is trivial as $a \to b \in S$ is term-definable as $(-a) + b \in B$. For the right to left implication, $+, 1$ are term-definable in $\langle \to, ;\rangle$ (see Proposition~\ref{prop:jointop}). By Lemma~\ref{lem:zeroRep} the $\langle \to, ; \rangle$-reduct of $\B$ is representable if and only if it has a representation where $h(0) = \emptyset$. See how in that representation $a \to 0$ (corresponding to $-a + 0 = -a$ in the Boolean semigroup) is represented as $\top \setminus a \cup \emptyset = \top \setminus a$.

As the (finite) representation decision problem is undecidable for Boolean semigroups, we conclude the same for implication semigroups.
\qed\end{proof}

\section{Problems}

In this section we outline some open problems. It follows from the results in Section~\ref{sec:undec} that the class of representable implication semigroups is not finitely axiomatisable, nor does it have the finite representation property, i.e. not every finite representable structure in the class is finitely representable. However, another decision problem of interest is posed below.

\begin{problem}
Is membership in the equational theory generated by the class of representable implication semigroups decidable?
\end{problem}

The reader can see that if we add the bottom element $0$ to the signature, the undecidability follows from the undecidability of the equational theory of the Boolean semigroups as described in \cite{hirsch2021undecidability}. This is because all negations of terms $-t$ can be rewritten as $t \to 0$ and all joins $t+t'$ as $(t \to t') \to t'$ where $t,t'$ are terms.

The problem remains open for the class of representable implication semigroups without the bottom element. One of the possible ways to prove undecidability is by using discriminator terms, defined below.

\begin{definition}
A \emph{discriminator term} $d(a,b,c)$ is a term defined in terms of elements of algebra $a,b,c$ such that for all representable algebras $d(a,b,c) = c$ if $a=b$ and $a$ otherwise.
\end{definition}

Although the existence of discriminator terms is not a guarantee for the undecidability of the equational theory membership decision problem, it is an interesting open question in its own right. 

\begin{problem}
Is it possible to define a discriminator term in the language of implication semigroups?
\end{problem}

It is well known that subreducts of representable relation algebras form quasivarieties. As such, the class of implication semigroups can be characterised by quasiequations. However, some open questions about the equational theory generated by the class of representable implication semigroups are listed below. 

\begin{problem}
Is the class of representable implication semigroups a variety?
\end{problem}

\begin{problem}
Is the equational theory generated by the class of representable implication semigroups finitely axiomatisable?
\end{problem}

We continue by looking at the alternative interpretations of $\to$ operation for binary relations. An interesting example, as mentioned in the introduction section is that of a weakening relation defined below.

\begin{definition}
Let $\mathbf{P} = (X, \leq)$ be a poset. $R \subseteq X\times X$ is a \emph{weakening relation} if and only if ${\leq};R;{\leq} \subseteq R$.
\end{definition}

In the context of the weakening relation algebras as described in \cite{galatos2020structure}, the $\to$ operation can be given in first order terms as
$$R \to S = \{(x,y) \mid \forall x',y': ((x' \leq x \wedge y \leq y' \wedge (x',y') \in R) \Rightarrow (x',y') \in S) \}$$
where $R,S$ are weakening relations over a poset $\mathbf{P} = (X,\leq)$.

This interpretation of the $\to$ operation gives rise to the class of representable weakening implication semigroups, for which the following properties remain open.

\begin{problem}
Is the (finite) representation decision problem decidable for the class of representable weakening implication semigroups? Is the class finitely axiomatisable and does it have the finite representation property?
\end{problem}

\begin{problem}
Is the class of representable weakening implication semigroups a (discriminator) variety? Is the equational theory generated by the class finitely axiomatisable/decidable?
\end{problem}
% Conjecture: every ISG can be represented as the implication semigroup of all endomorphisms of some ISG in G(T).

Finally, we note that it can be checked that all results presented in this paper can be generalised to the dual operation $\leftarrow$ by presenting dual axioms for the class of implication algebras and defining negation as $0 \leftarrow a$.

\bibliographystyle{splncs04.bst}
\bibliography{ref}

\end{document}